\documentclass[a4paper,11pt]{article}
\usepackage[latin1]{inputenc}
\usepackage{amsthm}
\usepackage{amsmath,amssymb,amsfonts}
\usepackage[mathscr]{euscript}
\usepackage{enumerate}
\usepackage{xspace}

\usepackage{url}
\numberwithin{equation}{section}
\newtheorem{thm}{Theorem}[section]

\newtheorem{prop}[thm]{Proposition}

\newtheorem{cor}[thm]{Corollary}
\newtheorem{Rq}[thm]{Remark}

\theoremstyle{definition}

\theoremstyle{remark}

\theoremstyle{plain}

\DeclareMathAlphabet{\calptmx}{OMS}{ztmcm}{m}{n}

\newcommand{\R}{\mathbb{R}}
\newcommand{\Sn}{\mathbb{S}^n}

\newcommand{\Cc}{\mathcal{C}}

\newcommand{\Sb}{\mathbb{S}}
\newcommand{\non}{\noindent}

\author{Bakri Laurent}
\title{ Quantitative uniqueness for Schr\"odinger operator with $\mathcal{C}^1$ potential}
\begin{document}

\date{}
\maketitle
 \begin{abstract}
 We give an upper bound on the vanishing order of solutions to Schr\"odinger equation on a compact smooth manifold.
 Our method is based on Carleman type inequalities, and gives a generalisation to a result of H. Donnely and C. Fefferman \cite{DF1} on eigenfunctions.
 \end{abstract}

\section{Introduction}
Let $(M,g)$ be a smooth,  compact and connected, $n$-dimensional Riemannian manifold.
It is well kown that, if $u$ is a non trivial solution of second order linear elliptic equation on $M$, then all zeros of $u$ are of finite order (\cite{aron,HS}).
The aim of this paper is to obtain quantitative estimate on the vanishing order of (non trivial) solutions to  \begin{equation}\label{sc}\Delta u=Wu,\end{equation}
 when $W\in\Cc^1(M)$.
In the particular case of eigenfunctions of the Laplacian  (\emph{i.e.} $W=\lambda$ is a constant), it has been shown by H. Donnelly and C. Fefferman \cite{DF1}
that the vanishing order is bounded by $C\sqrt{\lambda}$.
 In $\cite{K}$, I. Kukavica established some quantitative results for solution to \eqref{sc}. When $W$ is a bounded function he obtained that the vanishing order of solutions 
to \eqref{sc} is everywhere less than 
 $C(1+\sqrt{\|W\|_{\infty}}+\left(\mathrm{osc}\:(W)\right)^2)$, where $\mathrm{osc}(W)=\sup W-\inf W$ and $C$ a constant depending only on $M$. If $W$ is
 $\Cc^1$ he got the upper bound  $C(1+\|W\|_{\Cc^1})$, with $\|W\|_{\Cc^1}=\|W\|_{\infty}+\|\nabla W\|_{\infty}$. Our  main result is the following
\begin{thm}\label{van}
The vanishing order of solutions to \eqref{sc} is everywhere less than $$C_1\sqrt{\|W\|_{\Cc^1}}+C_2,$$
where $C_1$ and $C_2$ are positive constants depending only on $M$.
\end{thm}
More precisely theorem \ref{van} is a direct consequence of the following doubling inequality on solutions (theorem \ref{dor}) : 
\begin{equation}\label{doubl}\|u\|_{L^2(B_{2r}(x_0))}\leq e^{C_1\sqrt{\|W\|_{\Cc^1}}+C_2}\|u\|_{L^2(B_r(x_0))}.\end{equation}
The exponent $1/2$ on $\|W\|_{\Cc^1}$ in this result is sharp and agrees with the result of Donnelly and Fefferman \cite{DF1} when $W$ is constant. 
Indeed consider the homogeneous polynomials
 $f_k(x_1,x_2,\cdots,x_{n+1})=\Re e(x_1+ix_2)^k$ defined in $\R^{n+1}$. Set
$Y_k$ the restriction of $f_k$ to $\mathbb{S}^n$. $(Y_k)_k$ is  a sequence of spherical harmonics and
$-\Delta_{\Sn} Y_k=k(k+n-1)Y_k=\lambda_kY_k$. The vanishing order at the north pole $N=(0,\cdots,0,1)$ of $Y_k$ is $k\geq C\sqrt{\lambda_k}$.\newline
Let us now discuss briefly the methods usually used to deal with quantitative uniqueness for linear partial differential equations. They are two principal methods : the first one is based on  Carleman-type estimates
 \cite{aron,DF1,DF3,horm2,JK,JL} and the second one relies on the frequency function of solutions \cite{D,GL,K,L}.
The goal of  both methods is to control the local behaviour of solutions. 
In the original works of Donnelly and Fefferman \cite{DF1}, the authors wrote down a Carleman estimate on the operator $\Delta +\lambda$. 
Later, several authors (F-H Lin \cite{L}, Jerison-Lebeau \cite{JL}, Kukavika \cite{K2},...)  obtained some generalizations and simplifications in the proof. 
In particular, if $u$ is an eigenfunction of the Laplace operator on $M$, with eigenvalue $\lambda$, then the function $\tilde{u}$ defined on $M\times[-T,T]$ by
$$\tilde{u}(x,t)=\cosh(\sqrt{\lambda}t)u(x)$$ satisifies $(\Delta  +\frac{\partial^2}{\partial t^2})\tilde{u}=0$. 
The problem is then simplified since one has only to deal with the $0$ eigenvalue
 of the operator $\Delta +\frac{\partial^2}{\partial t^2}$ on $M\times[-T,T]$.
 By example, in \cite{JL},  D. Jerison and G. Lebeau  established a Carleman estimate on $\Delta +\frac{\partial^2}{\partial t^2}$.
However, it was pointed out by Kukavica  \cite{K} that, 
the method of \cite{JL}  doesn't seems to extend easily when studying the more general equation \eqref{sc}.
Despite this, the point of our paper is that one can sucessfully establish a Carleman estimate directly on the operator $\Delta +W$. 
Furthermore it leads to a better upper bound on the vanishing order of solutions to \eqref{sc} (for $W\in\Cc^1$). 
\newline
The paper is organised as follows. In section 2 we establish Carleman estimate for the operator $\Delta  +W$. 
In section 3 we deduce, in a standard manner, three balls theorem for solutions to \eqref{sc}, then using compactness we derive doubling 
inequality which gives immediatly theorem \ref{van}.
In a fourthcoming paper we study the vanishing order of solutions when $W$ is only a bounded function.\vspace{0,5cm}\newline 
\section{Carleman estimates} 

\non Fix $x_0$ in $M$,  and let : $r=r(x)=d(x,x_0)$ the Riemannian distance from $x_0$. We denote by $B_r(x_0)$ the geodesic ball centered at  $x_0$ of radius $r$.
We will denote  by $\|\cdot\|$ the  $L^2$ norm. 
Recall that Carleman estimates are weighted integral inequalities with a weight function $e^{\tau\phi}$, where the function 
$\phi$ satisfy some convexity properties. 
Let us now define the weight function we will use.\newline  
For a fixed number $\varepsilon$ such that  $0<\varepsilon<1$ and $T_0<0$, we define the function $f$ on  $]-\infty,T_0[ $ by $f(t)=t-e^{\varepsilon t}$. 
One can check easily that,  for $|T_0|$ great enough, the function $f$ verifies the following properties:
\begin{equation}\label{f}
\begin{array}{lcr}
& 1-\varepsilon e^{\varepsilon T_0}\leq f^\prime(t)\leq 1&\forall t\in]-\infty,T_0[,\\
&\displaystyle{\lim _{t\rightarrow-\infty}-e^{-t}f^{\prime\prime}(t)}=+\infty. &
\end{array}
\end{equation}
Finally we define $\phi(x)=-f(\ln r(x))$. Now we can state the main result of this section:

\begin{thm}\label{tics}
 There exist positive constants $R_0, C,C_1,C_2$, which depend only on  $M$, such that, 
for any $\:W\in \Cc^1(M)$, $x_0\in M$, $u\in C^\infty_0(B_{R_0}(x_0)\setminus\{0\})$ and 
 $\tau \geq C_1\sqrt{\|W\|_{\mathcal{C}^1}}+C_2$, one has    
\begin{equation}\label{S}
C\left\|r^2e^{\tau\phi}\left(\Delta u +Wu \right)\right\|\geq \tau^\frac{3}{2}\left\|r^{\frac{\varepsilon}{2}}e^{\tau\phi}u\right\|
+ \tau^{\frac{1}{2}}\left\|r^{1+\frac{\varepsilon}{2}}e^{\tau\phi}\nabla u\right\|.
\end{equation}

\non Moreover, if  $$\mathrm{supp}(u)\subset\{x\in M; r(x)\geq\delta>0\},$$ then
\begin{equation}\label{Siv2}\begin{array}{rcc}
C\left\|r^2e^{\tau\phi}\left(\Delta u +Wu \right)\right\|&\geq& \tau^\frac{3}{2}\left\|r^{\frac{\varepsilon}{2}}e^{\tau\phi}u\right\| \\
+\ \tau\delta\left\|r^{-1}e^{\tau\phi}u\right\|&+& \tau^{\frac{1}{2}}\left\|r^{1+\frac{\varepsilon}{2}}e^{\tau\phi}\nabla u\right\|.
\end{array}\end{equation}
\end{thm}
\begin{Rq}
 This inequality can be seen as a generalization of previous Carleman type estimates in the case that $W$ is a constant (see \cite{DF1}).
 Indeed when $W=\lambda$ one has $\sqrt{\|W\|}_{\Cc^1}=\sqrt{\lambda}$.
 The point is that since $W$ is $\Cc^1$ we will be allowed to integrate by parts, but then we have to take care of the derivatives of $W$.
  \vspace{0,4mm}
\end{Rq}
\begin{Rq}
 In the inequalities \eqref{S} and \eqref{Siv2} the gradient terms are not necessary to the purpose of this paper. 
We choose to include them for a more general statement.
\end{Rq}

\begin{proof}Hereafter $C$, $C_1$, $C_2$  and $c$ denote positive constants depending only upon $M$, though their values may change from one line to another. 
Without loss of generality, we  may suppose that all functions are real.   
We now introduce the polar geodesic coordinates $(r,\theta)$ near $x_0$. Using Einstein notation, the Laplace operator takes the form : 
$$r^2\Delta u=r^2\partial_r^2u+r^2\left(\partial_r\ln(\sqrt{\gamma})+\frac{n-1}{r}\right)\partial_ru+
\frac{1}{\sqrt{\gamma}}\partial_i(\sqrt{\gamma}\gamma^{ij}\partial_ju),$$
where $\displaystyle{\partial_i=\frac{\partial}{\partial\theta_i}}$ and for each fixed $r$,  $\ \gamma_{ij}(r,\theta)$  
is a metric on \: $\Sb^{n-1}$ and $\displaystyle{\gamma=\mathrm{det}(\gamma_{ij})}$.\newline
Since $(M,g)$ is smooth, we have for $r$ small enough :  
\begin{eqnarray}\label{m1}
\partial_r(\gamma^{ij})&\leq& C (\gamma^{ij})\ \ \ \mbox{(in the sense of tensors)}; 
\nonumber \\
|\partial_r(\gamma)|&\leq& C;\\
C^{-1}\leq\gamma&\leq& C.\nonumber 
 \end{eqnarray}

\non Set  $r=e^t$, we have $\displaystyle{\frac{\partial}{\partial r}=e^{-t}\frac{\partial}{\partial t}}$. 
Then the function $u$ is supported in \\ $]-\infty,T_0[\times\mathbb{S}^{n-1},$ 
where $|T_0|$ will be chosen large enough. In this new variables, we can write : 
$$e^{2t}\Delta u=\partial_t^2u+(n-2+\partial_t\mathrm{ln}\sqrt{\gamma})\partial_tu
+\frac{1}{\sqrt{\gamma}}\partial_i(\sqrt{\gamma}\gamma^{ij}\partial_ju).$$
The conditions (\ref{m1}) become
\begin{eqnarray}\label{m2}
\partial_t(\gamma^{ij})&\leq& Ce^t (\gamma^{ij})\nonumber\ \ \ \mbox{(in the sense of tensors)};  \\
|\partial_t(\gamma)|&\leq& Ce^t;\\
C^{-1}\leq\gamma &\leq & C. \nonumber
\end{eqnarray}
Now we introduce the conjugate operator : 
\begin{equation}\begin{array}{rcl}
L_\tau(u)&=&e^{2t}e^{\tau\phi}\Delta(e^{-\tau\phi}u)+e^{2t}Wu\\
               &=&\partial^2_tu+\left(2\tau f^\prime+n-2+\partial_t\mathrm{ln}\sqrt{\gamma}\right)\partial_tu\\
               &+&\left(\tau^2f^{\prime^2}+\tau f^{\prime\prime}+(n-2)\tau f^{\prime}+\tau\partial_t\mathrm{ln}\sqrt{\gamma}f^{\prime}\right)u\\
               &+&\Delta_\theta u+e^{2t}Wu,
     \end{array}\end{equation}
with $$\Delta_\theta u=\frac{1}{\sqrt{\gamma}}\partial_i\left(\sqrt{\gamma}\gamma^{ij}\partial_ju\right).$$

\noindent It will be useful for us to introduce the following $L^2$ norm on $]-\infty,T_0[\times\Sb^{n-1} $: 
$$\|V\|_f^2=\int_{]-\infty,T_0[\times\Sb^{n-1}} V^2\sqrt{\gamma}{f^{\prime}}^{-3}dtd\theta,$$ where $d\theta$ is the usual measure on $\Sb^{n-1}$.
The corresponding inner product is denoted by  $\left\langle\cdot,\cdot\right\rangle_f$ , \emph{i.e} $$\langle u,v\rangle_f = \int uv\sqrt{\gamma}{f^{\prime}}^{-
3}dtd\theta.$$ 
\noindent We will estimate from below $\|L_\tau u\|^2_f$ by using elementary algebra and integrations by parts. We are concerned, in the computation, by the power of $\tau$ 
and  exponenial decay when $t$ goes to $-\infty$. First by triangular inequality one has
\begin{equation}
\|L_\tau(u)\|_f\geq I-I\!I ,
\end{equation}
with 
\begin{equation}\begin{array}{rcl}
I&=& \left\|\partial^2_tu+2\tau f^\prime\partial_tu+\tau^2f^{\prime^2}u+e^{2t}Wu+\Delta_\theta u\right\|_f,\\
I\!I&=&\left\|\tau f^{\prime\prime}u+(n-2)\tau f^\prime u+\tau\partial_t\mathrm{ln}\sqrt{\gamma}f^{\prime}u\right\|_f\\
&+&\left\|(n-2)\partial_tu+\partial_t\ln\sqrt{\gamma}\partial_tu\right\|_f.
\end{array}\end{equation}
We will be able to absorb $I\!I$ later. Then we compute $I^2$ :

$$I^2=I_1+I_2+I_3,$$ with 

\begin{equation}\begin{array}{rcl}
I_1&=&\|\partial^2_tu+(\tau^2f^{\prime^2}+e^{2t}W)u+\Delta_\theta u\|_f^2\\
I_2&=&\|2\tau f^\prime\partial_tu\|_f^2\\
I_3&=&2\left<2\tau f^{\prime}\partial_tu ,\partial^2_tu+\tau^2f^{\prime^2}u+e^{2t}Wu+\Delta_\theta u\right>_f\label{I}
\end{array}\end{equation}
           
\non In order to compute $I_3$ we write it in a convenient way: 
\begin{equation}
 I_3=J_1+J_2+J_3,
\end{equation}
where the integrals $J_i$ are defined by :
\begin{equation}\label{I3}
\begin{array}{rcl}
J_1&=&2\tau \int f^\prime \partial_t(|\partial_tu|^2)f^{\prime^{-3}}\sqrt{\gamma}dtd\theta\\
J_2&=&4\tau\int f^\prime\partial_tu\partial_i\left(\sqrt{\gamma}\gamma^{ij}\partial_ju\right)f^{\prime^{-3}}dtd\theta\\
J_3&=&\int\left(2\tau^3 (f^\prime)^3+2\tau f^\prime e^{2t}W\right)2u\partial_tuf^{\prime^{-3}}\sqrt{\gamma}dtd\theta.
\end{array}
\end{equation}
\noindent Now we will use integration by parts to estimate each terms of \eqref{I3}. 
Note that $f$ is radial and that $2\partial_tu\partial_t^2u=\partial_t(|\partial_tu|^2)$. We find that  :
\begin{equation*}\begin{array}{rcl}
J_1&=& \int\left(4\tau f^{\prime\prime}\right)|\partial_tu|^2f^{\prime^{-3}}\sqrt{\gamma}dtd\theta\\
&-&\int2\tau f^\prime\partial_t\mathrm{ln}\sqrt{\gamma}|\partial_{t}u|^2f^{\prime^{-3}}\sqrt{\gamma}dtd\theta.\end{array}
\end{equation*}
The conditions \eqref{m2} imply that $|\partial_t\ln\sqrt{\gamma}|\leq Ce^t$. Then properties \eqref{f} on $f$ gives, for large $|T_0|$ 
that $|\partial_t\ln\sqrt{\gamma}|$ is small compared to $|f^{\prime\prime}|$. Then one has   
\begin{equation}\label{J_1}J_1\geq -c\tau\int |f^{\prime\prime}|\cdot|\partial_tu|^2f^{\prime^{-3}}\sqrt{\gamma}dtd\theta.\end{equation}

\noindent Now in order to estimate $J_2$ we first integrate by parts with respect to $\partial_i$ : 
\begin{equation*}\begin{array}{rcl}
J_2
   &=&-2\int2\tau f^{\prime}\partial_t\partial_iu\gamma^{ij}\partial_juf^{\prime^{-3}}\sqrt{\gamma}dtd\theta.                  
     \end{array}
\end{equation*}            
Then we integrate by parts with respect to $\partial_t$. We get : 
\begin{equation*}\begin{array}{rcl}
J_2&=&-4\tau\int f^{\prime\prime}\gamma^{ij}\partial_iu\partial_juf^{\prime^{-3}}\sqrt{\gamma}dtd\theta\\
&+&\int2\tau f^{\prime}\partial_t\mathrm{ln}\sqrt{\gamma}\gamma^{ij}\partial_iu\partial_juf^{\prime^{-3}}\sqrt{\gamma}dtd\theta\\
&+&\int2\tau f^\prime\partial_t(\gamma^{ij})\partial_iu\partial_juf^{\prime^{-3}}\sqrt{\gamma}dtd\theta.
\end{array}
\end{equation*}
We denote  $|D_\theta u|^2=\partial_iu\gamma^{ij}\partial_ju$. Now using that $-f^{\prime\prime}$ is non-negative and $\tau$ is large, 
the conditions \eqref{f}  and \eqref{m2} gives for $|T_0|$ large enough:
\begin{equation}\label{J_2}
 J_2\geq 3\tau\int|f^{\prime\prime}|\cdot|D_\theta u|^2f^{\prime^{-3}}\sqrt{\gamma}dtd\theta.
\end{equation}
Similarly computation of $J_3$ gives :
\begin{equation}\label{J_31}\begin{array}{rcl}
J_3&=&-2\int\tau^3\partial_t\mathrm{ln}(\sqrt{\gamma})u^2\sqrt{\gamma}dtd\theta\\
&-&\int(4f^\prime-4 f^{\prime\prime}+2f^{\prime}\partial_t\ln\sqrt{\gamma} )\tau e^{2t}Wu^2f^{\prime^{-3}}\sqrt{\gamma}dtd\theta\\
&-&\int2\tau f^\prime e^{2t}\partial_tW|u|^2f^{\prime^{-3}}\sqrt{\gamma}dtd\theta.
\end{array}\end{equation}
\noindent Now we assume that \begin{equation}\label{tau}\tau\geq C_1\sqrt{\|W\|_{\mathcal{C}^1}}+C_2.\end{equation}
 From  \eqref{f} and \eqref{m2} one can see that if $C_1$, $C_2$ and $|T_0|$ are large enough, then\:\newline
\begin{equation}\label{J_3}
J_3\geq-c\tau^3\int e^t|u|^2f^{\prime^{-3}}\sqrt{\gamma}dtd\theta.
\end{equation}

 \noindent Thus far, using \eqref{J_1},\eqref{J_2} and \eqref{J_3}, we have : 
 \begin{equation}\label{I_3}\begin{array}{rcl}
 I_3 &\geq & 3\tau \int\left|f^{\prime\prime}\right|\left|D_\theta u\right|^2f^{\prime^{-3}}\sqrt{\gamma}dtd\theta -c\tau^3\int e^t|u|^2f^{\prime^{-3}}\sqrt{\gamma}dtd\theta\\
 &-&c\tau\int\left|f^{\prime\prime}\right|\left|\partial_t u\right|^2f^{\prime^{-3}}\sqrt{\gamma}dtd\theta.
 \end{array}\end{equation}

\noindent Now we consider  $I_1$ :
$$I_1=\left\|\partial^2_tu+\left(\tau^2
f^{\prime^2}+e^{2t}W\right)u+\Delta_\theta u\right\|_f^2 .$$
Let $\rho>0$ a small number to be chosen later. Since  $|f^{\prime\prime}|\leq1$ and $\tau\geq1$, we have : \newline
\begin{equation}I_1\geq\frac{\rho}{\tau}I_1^\prime,\end{equation}
where $I_1^\prime$ is defined by :
\begin{equation}I_1^\prime=\left\|\sqrt{|f^{\prime\prime}|}\left[\partial^2_tu+\left(\tau^2f^{\prime^2}+e^{2t}W\right)u+\Delta_\theta u\right]\right\|_f^2 \end{equation}
and one has \begin{equation}I_1^\prime=K_1+K_2+K_3,\end{equation} 
with 
\begin{equation}\label{Ki}
\begin{array}{rcl}
  K_1&=&\left\|\sqrt{|f^{\prime\prime}|}\left(\partial_t^2u+\Delta_\theta u\right)\right\|_f^2, \\
K_2&=&\left\|\sqrt{|f^{\prime\prime}|}\left(\tau^2f^{\prime^2}+e^{2t}W\right)u\right\|_f^2,  \\
K_3&=&2\left\langle\left(\partial_t^2u+\Delta_\theta u\right)\left|f^{\prime\prime}\right|,\left(\tau^2f^{\prime^2}+e^{2t}W\right)u\right\rangle_f.
\end{array}\end{equation}
Integrating by parts gives : 
\begin{equation}\label{K_3}\begin{array}{rcl}
K_3&=&2\int f^{\prime\prime}\left(\tau^2f^{\prime^2}+e^{2t}W\right)|\partial_tu|^2f^{\prime^{-3}}\sqrt{\gamma}dtd\theta\\
&+&2\int \partial_t\left[f^{\prime\prime}\left(\tau^2f^{\prime^2}+e^{2t}W \right)\right]\partial_tuu\sqrt{\gamma}f^{\prime^{-3}}dtd\theta \\
&-&6\int \left(f^{\prime\prime^2}f^{\prime^{-1}}\left(\tau^2f^{\prime^2}+e^{2t}W \right)\right)\partial_tuu\sqrt{\gamma}f^{\prime^{-3}}dtd\theta \\
&+&2\int f^{\prime\prime}\left(\tau^2f^{\prime^2}+e^{2t}W\right)\partial_t\mathrm{ln}\sqrt{\gamma}\partial_tuuf^{\prime^{-3}}\sqrt{\gamma}dtd\theta\\
&+&2\int f^{\prime\prime}\left(\tau^2f^{\prime^2}+e^{2t}W\right)|D_\theta u|^2f^{\prime^{-3}}\sqrt{\gamma}dtd\theta\\
&+&2\int f^{\prime\prime}e^{2t}\partial_iW\cdot\gamma^{ij}\partial_juuf^{\prime^{-3}}\sqrt{\gamma}dtd\theta.
\end{array}\end{equation}
 The condition $\tau\geq C_1\sqrt{\|W\|_{\Cc^1}}+C_2$ implies, $$|\partial_iW\gamma^{ij}\partial_juu|\leq c\tau^2(|D_\theta u|^2+|u|^2).$$
Now since  $2\partial_tuu\leq u^2+|\partial_tu|^2$, we can  use conditions \eqref{f} and \eqref{m2} to get

\begin{equation}
 K_3\geq-c\tau^2\int|f^{\prime\prime}|\left(|\partial_tu|^2+|D_\theta u|^2+|u|^2\right)f^{\prime^{-3}}\sqrt{\gamma}dtd\theta\\
\end{equation}
We also have 

\begin{equation}
K_2\geq c\tau^4\int|f^{\prime\prime}||u|^2f^{\prime^{-3}}\sqrt{\gamma}dtd\theta
\end{equation}
and since  $K_1\geq 0$ ,  
\begin{equation}\label{I_1}\begin{array}{rcl}
I_1&\geq&-\rho c\tau\int|f^{\prime\prime}|\left(|\partial_t u|^2+|D_\theta u|^2\right)f^{\prime^{-3}}\sqrt{\gamma}dtd\theta\\
&+&C\tau^{3}\rho\int|f^{\prime\prime}||u|^2f^{\prime^{-3}}\sqrt{\gamma}dtd\theta.
\end{array}\end{equation}

\noindent Then using  \eqref{I_3} and \eqref{I_1} 
\begin{equation}\label{pro}\begin{array}{rcl}
I^2  &\geq & 4\tau^2\| f^\prime\partial_tu\|_f^2+3\tau\int|f^{\prime\prime}||D_\theta u|^2f^{\prime^{-3}}\sqrt{\gamma}dtd\theta \\
&+&C\tau^{3}\rho\int|f^{\prime\prime}||u|^2f^{\prime^{-3}}\sqrt{\gamma}dtd\theta-c\tau^3\int e^t|u|^2f^{\prime^{-3}}\sqrt{\gamma}dtd\theta\\
&-&\rho c\tau\int|f^{\prime\prime}|\left(|u|^2+|\partial_tu|^2+|D_\theta u|^2\right)f^{\prime^{-3}}\sqrt{\gamma}dtd\theta.\\
&-&c\tau\int|f^{\prime\prime}||\partial_tu|^2f^{\prime^{-3}}\sqrt{\gamma}dtd\theta
\end{array}\end{equation}

\noindent
Now one needs to check that every non-positive term in the right hand side of \eqref{pro} can be absorbed in the first three terms. 
\\ First fix $\rho$ small enough such that
$$\rho c\tau\int|f^{\prime\prime}|\cdot|D_\theta u|^2{f^{\prime}}^{-3}\sqrt{\gamma}dtd\theta\leq 2\tau\int|f^{\prime\prime}|\cdot|D_\theta u|^2f^{\prime^{-3}}\sqrt{\gamma}dtd\theta$$
where $c$ is the constant appearing in \eqref{pro}. The other  terms in the last integral of \eqref{pro} can then be absorbed by comparing powers of $\tau$ 
(for $C_2$ large enough). 
Finally since conditions \eqref{f} imply that $e^t$ is small compared to $|f^{\prime\prime}|$, 
we can absorb   $-c\tau^3e^t|u|^2$ in $C\tau^{3}\rho|f^{\prime\prime}||u|^2$.

\noindent Thus we obtain :
\begin{equation}\label{ssu}\begin{array}{rcl}
 I^2 &\geq &C\tau^2\int|\partial_t u|^2f^{\prime^{-3}}\sqrt{\gamma}dtd\theta+C\tau\int|f^{\prime\prime}||D_\theta u|^2f^{\prime^{-3}}\sqrt{\gamma}dtd\theta\\
&+ &C\tau^{3}\int|f^{\prime\prime}||u|^2f^{\prime^{-3}}\sqrt{\gamma}dtd\theta
\end{array}\end{equation}

\noindent As before, we can check that $I\!I$ can be absorbed in $I$ for  $|T_0|$ and $\tau$ large enough. 
Then we obtain \begin{equation}\label{bla}\|L_\tau u\|_f^2\geq C\tau^3\|\sqrt{|f^{\prime\prime}|} u\|_f^2+C\tau^2\|\partial_t u\|_f^2
+C\tau\|\sqrt{|f^{\prime\prime}|}D_\theta u\|_f^2 .\end{equation}
Note that, since $\tau$ is large and $\sqrt{|f^{\prime\prime}|}\leq1$, one has
 \begin{equation}\|L_\tau u\|_f^2\geq C\tau^3\|\sqrt{|f^{\prime\prime}|} u\|_f^2+c\tau\|\sqrt{|f^{\prime\prime}|}\partial_t u\|_f^2
+C\tau\|\sqrt{|f^{\prime\prime}|}D_\theta u\|_f^2,\end{equation}
and the constant $c$ can be choosen arbitrary smaller than $C$. 
If we set  $v=e^{-\tau\phi}u$, then we have 
\begin{equation*}
 \begin{array}{rcl}\|e^{2t}e^{\tau\phi}(\Delta v+Wv)\|_f^2&\geq& C\tau^3\|\sqrt{|f^{\prime\prime}|}e^{\tau\phi} v\|_f^2
-c\tau^3\|\sqrt{|f^{\prime\prime}|}f^\prime e^{\tau\phi} v\|_f^2\\
+
\frac{c}{2}\tau\|\sqrt{|f^{\prime\prime}|}e^{\tau\phi}\partial_t v\|_f^2&+&C\tau\|\sqrt{|f^{\prime\prime}|}e^{\tau\phi} D_\theta v\|_f^2\end{array}.
\end{equation*}
Finally since $f^\prime$ is close to 1 one can absorb the negative term to obtain
 
\begin{equation}\begin{array}{rcl}\|e^{2t}e^{\tau\phi}(\Delta v+Wv)\|_f^2&\geq& C\tau^3\|\sqrt{|f^{\prime\prime}|}e^{\tau\phi} v\|_f^2\\
+
C\tau\|\sqrt{|f^{\prime\prime}|}e^{\tau\phi}\partial_t v\|_f^2&+&C\tau\|\sqrt{|f^{\prime\prime}|}e^{\tau\phi} D_\theta v\|_f^2\end{array}.\end{equation}
It remains to get back to the usual $L^2$ norm. First note that since $f^\prime$ is close to 1 \eqref{f},
 we can get the same estimate without the term $(f^\prime)^{-3}$ in the integrals. Recall that in polar coordinates $(r,\theta)$ the volume element 
is $r^{n-1}\sqrt{\gamma}drd\theta$, we can deduce from \eqref{ssu} by substitution that : 
\begin{equation}
 \begin{array}{rcl}
  \|r^2e^{\tau\phi}(\Delta v+Wv)r^{-\frac{n}{2}}\|^2&\geq&C\tau^3\|r^\frac{\varepsilon}{2}e^{\tau\phi}vr^{-\frac{n}{2}}\|^2\\
&+&C\tau\|r^{1+\frac{\varepsilon}{2}}e^{\tau\phi}\nabla vr^{-\frac{n}{2}}\|^2.
 \end{array}
\end{equation}
Finally one can get rid of the term $r^{-\frac{n}{2}}$ by replacing $\tau$ with $\tau+\frac{n}{2}$. Indeed from 
$e^{\tau\phi}r^{-\frac{n}{2}}=e^{(\tau+\frac{n}{2})\phi}e^{-\frac{n}{2}r^\varepsilon}$ one can check easily that, for $r$ small enough  
$$\frac{1}{2}e^{(\tau+\frac{n}{2})\phi}\leq e^{\tau\phi}r^{-\frac{n}{2}}\leq e^{(\tau+\frac{n}{2})\phi}.$$ 

\non This achieves the  proof of the first part of theorem \ref{tics}.\\

\non Now suppose that  $\mathrm{supp}(u)\subset\{x\in M; r(x)\geq\delta>0\}$ and define $T_1=\ln\delta$.\newline

\noindent Cauchy-Schwarz inequality apply to $$\int\partial_t(u^2)e^{-t}\sqrt{\gamma}dtd\theta=2\int u\partial_tue^{-t}\sqrt{\gamma}dtd\theta$$ 
gives
\begin{equation}\label{d1}
 \int\partial_t(u^2)e^{-t}\sqrt{\gamma}dtd\theta\leq 2\left(\int\left(\partial_tu\right)^2e^{-t}\sqrt{\gamma}dtd\theta \right)^{\frac{1}{2}}
\left(\int u^2e^{-t}\sqrt{\gamma}dtd\theta\right)^{\frac{1}{2}}.
\end{equation}
On the other hand, integrating by parts gives
 \begin{equation}
      \int\partial_t(u^2)e^{-t}\sqrt{\gamma}dtd\theta = \int u^2e^{-t}\sqrt{\gamma}dtd\theta-\int u^2e^{-t}\partial_t(\ln(\sqrt{\gamma}))\sqrt{\gamma}dtd\theta.
     \end{equation}
Now since  $|\partial_t\ln\sqrt{\gamma}|\leq Ce^t$ for $|T_0|$ large enough we can deduce :
\begin{equation}\label{d2}
 \int\partial_t(u^2)e^{-t}\sqrt{\gamma}dtd\theta \geq c  \int u^2e^{-t}\sqrt{\gamma}dtd\theta.
\end{equation}
Combining \eqref{d1} and  \eqref{d2} gives  
\begin{eqnarray*}\label{prout}
 c^2 \int u^2e^{-t}\sqrt{\gamma}dtd\theta&\leq& 4\int\left(\partial_tu\right)^2e^{-t}\sqrt{\gamma}dtd\theta\\
&\leq&4e^{-T_1}\int\left(\partial_t u\right)^2\sqrt{\gamma}dtd\theta.\end{eqnarray*}

\noindent Finally, droping all terms except  $\tau^2\int|\partial_t u|^2f^{\prime^{-3}}\sqrt{\gamma}dtd\theta$  in \eqref{ssu}  gives :

\begin{eqnarray*}
C^{\prime}I^2\geq \tau^2 \delta^2 \|e^{-t}u\|_f^2.
\end{eqnarray*}\vspace{0,5cm}
Inequality \eqref{ssu} can then be replaced by :  
\begin{equation}\begin{array}{rcl}
 I^2 &\geq &C\tau^2\int|\partial_t u|^2f^{\prime^{-3}}\sqrt{\gamma}dtd\theta+C\tau\int|f^{\prime\prime}|\cdot|D_\theta u|^2f^{\prime^{-3}}\sqrt{\gamma}dtd\theta\\
&+ &C\tau^{3}\int|f^{\prime\prime}|\cdot|u|^2f^{\prime^{-3}}\sqrt{\gamma}dtd\theta+C\tau^2 \delta^2\int|u|^2{f^\prime}^{-3}\sqrt{\gamma}dtd\theta.
\end{array}\end{equation}
The rest of the proof follows in a similar way than the first part. 
\end{proof}

\section{Doubling inequality}
In this section we prove a doubling inequality for solutions of \eqref{sc}. 
First we deduce from Carleman estimate a three balls theorem for solutions. 
The standard way to do so is to apply such estimate, 
to $\psi u$ where $\psi$ is an appropriate cut off function and $u$ a solution, and make a good choice of the parameter $\tau$ (see \cite{JL}).
 We give a proof, following the method of Donnely and Fefferman \cite{DF1}, adapted to our choice of weight functions in the Carleman estimate.  
\begin{prop}[Three balls inequality]\label{tst}
There exist positive constants $R_1$, $C_1$, $C_2$ and $0<\alpha<1$  wich depend only on $M$ such that, if $u$ is a solution to \eqref{sc} with $W$ of class $\Cc^1$,  
then for any $R<R_1$, and any $x_0\in M, $ one has 
\begin{equation}\label{ttc}
\| u\|_{B_R(x_0)}\leq e^ {C_1 \sqrt{\|W\|_{\Cc^1}}+C_2}\|u\|_{B_{\frac{R}{2}}(x_0)}^\alpha\| u\|_{B_{2R}(x_0)}^{1-\alpha}.
\end{equation}
\end{prop}

\begin{proof}
Let $x_0$ a point in $M$. Let $u$ be a solution 
to \eqref{sc} and $R$ such that $0<R<\frac{R_0}{2}$ with $R_0$ as in theorem \ref{tics}. Recall that $r(x)$ is the riemannian distance between $x$ and $x_0$ and $B_r$ the geodesic ball centered at $x_0$ of radius $r$. 
If $v$ is a  function defined in a neigborhood of $x_0$, we denote by  $\|v\|_R$ the $L^2$ norm of $v$ on $B_R$ and by $\|v\|_{R_1,R_2}$ the $L^2$ norm of $v$ on the set  
 $A_{R_1,R_2}:=\{x\in M ;\:R_1\leq r(x)\leq R_2\}$.  
Let $\psi\in\Cc^{\infty}_0(B_{2R})$, $0\leq\psi\leq1$, a function with the following properties:
\begin{itemize}
\item[$\bullet$] $\psi(x)=0$ if $r(x)<\frac{R}{4}$ or  $r(x)>\frac{5R}{3}$,
\item[$\bullet$] $\psi(x)=1$ if $\frac{R}{3}<r(x)<\frac{3R}{2}$,
\item[$\bullet$] $|\nabla\psi(x)|\leq \frac{C}{R}$,
 \item[$\bullet$] $|\nabla^{2}\psi(x) |\leq \frac{C}{R^2}$.
\end{itemize}
First since the function $\psi u$ is supported in the annulus $A_{\frac{R}{3},\frac{5R}{3}}$ we can apply estimate \eqref{Siv2} of theorem \ref{tics}. 
In particular we have  :
\begin{equation}\label{debut}C\left\|r^2e^{\tau\phi}\left(\Delta \psi u+2\nabla u\cdot\nabla\psi\right)\right\|\geq\tau\left\|e^{\tau\phi}\psi u\right\|.\end{equation}

\non Assume that $\tau \geq 1$, and use properties of $\psi$ to get  : 
\begin{equation}\label{3c1}\begin{array}{rcl}
\|e^{\tau\phi}u\|_{\frac{R}{3},\frac{3R}{2}} &\leq  &C\left(\|e^{\tau\phi}u\|_{\frac{R}{4},\frac{R}{3}}+\|e^{\tau\phi}u\|_{\frac{3R}{2},\frac{5R}{3}}\right) \\
&+&C\left(R\|e^{\tau\phi}\nabla u\|_{\frac{R}{4},\frac{R}{3}}+R\|e^{\tau\phi}\nabla u\|_{\frac{3R}{2},\frac{5R}{3}}\right).\end{array}\end{equation}

\non Recall that  $\phi(x)=-\ln r(x)+r(x)^\varepsilon$. In particular $\phi$ is radial and decreasing (for small $r$). Then one has, 
$$\begin{array}{rcl}
\|e^{\tau\phi}u\|_{\frac{R}{3},\frac{3R}{2}}&\leq&
 C\left(e^{\tau\phi(\frac{R}{4})}\|u\|_{\frac{R}{4},\frac{R}{3}}+e^{\tau\phi(\frac{3R}{2})}\|u\|_{\frac{3R}{2},\frac{5R}{3}}\right)\\&+&C\left(Re^{\tau\phi(\frac{R}{4})}\|\nabla u\|_{\frac{R}{4},\frac{R}{3}}+Re^{\tau\phi(\frac{3R}{2})}\|\nabla u\|_{\frac{3R}{2},\frac{5R}{3}}\right).
\end{array}$$
Now we recall the following elliptic estimates : since $u$ satisfies \eqref{sc} then it is not hard to see that :
 \begin{equation}\label{lm1}\|\nabla u\|_{(1-a)r}\leq C\left(\frac{1}{(1-a)R}+\|W\|^{1/2}_{\infty}\right)\|u\|_{B_R}, \:\:\ \mathrm{for}\  \:0<a<1.   \end{equation}
Moreover since $A_{R_1,R_2}\subset B_{R_2}$, using formula \eqref{lm1} and properties of $\phi$ gives 

$$e^{\tau\phi(\frac{3R}{2})}\| u\|_{\frac{3R}{2},\frac{5R}{3}}\leq C\left(\frac{1}{R}+\|W\|^{1/2}_{\infty}\right)e^{\tau\phi(\frac{3R}{2})}\|u\|_{2R}.$$
Using  (\ref{3c1}) one has :
 \begin{equation*}\label{ind}
\|u\|_{\frac{R}{3},R} \leq C(\|W\|^{1/2}_{\infty}+1)\left( e^{\tau(\phi(\frac{R}{4})-\phi(R))}\|u\|_{\frac{R}{2}}+e^{\tau(\phi(\frac{3R}{2})-\phi(R))}\|u\|_{2R}\right).
  \end{equation*}
  Let $A_R=\phi(\frac{R}{4})-\phi(R)$ and  $B_R=-(\phi(\frac{3R}{2})-\phi(R))$.
  From the properties of  $\phi$, we have $0<A^{-1}\leq A_R\leq A$ and $0<B\leq B_R\leq B^{-1}$ where $A$ and $B$ don't depend on $R$. 
We may assume that $C(\|W\|^{1/2}_{\infty}+1)\geq 2$. Then we can add $\|u\|_{\frac{R}{3}}$ to each member and bound it in the right hand side by  
$C(\|W\|^{1/2}_{\infty}+1)e^{\tau A}\|u\|_{\frac{R}{2}}$. We get :
  
\begin{equation}\label{3c2}
\|u\|_{R} \leq  C(\|W\|^{1/2}_{\infty}+1)\left(e^{\tau A}\|u\|_{\frac{R}{2}}+e^{-\tau B}\|u\|_{2R}\right).
  \end{equation} 
Now we want to find $\tau$ such that $$C(\|W\|^{1/2}_{\infty}+1)e^{-\tau B}\|u\|_{2R}\leq \frac{1}{2}\|u\|_{R}$$
wich is true for $\tau \geq -\frac{1}{B}\ln\left(\frac{1}{2C(\|W\|^{1/2}_{\infty}+1)}\frac{\|u\|_R}{\|u\|_{2R}}\right).$ 
Since $\tau$ must also satisfy  $$\tau \geq C_1\sqrt{\|W\|_{\Cc^1}}+C_2,$$
we choose
\begin{equation} \label{tau2}
\tau = -\frac{1}{B}\ln\left(\frac{1}{2C(\|W\|^{1/2}_{\infty}+1)}\frac{\|u\|_R}{\|u\|_{2R}}\right)+C_1\sqrt{\|W\|_{\Cc^1}}+C_2.
\end{equation}

\noindent Since $\|W\|_\infty\leq\|W\|_{\Cc^1}$ we can deduce from \eqref{3c2} that :

\begin{equation}\label{last}
\|u\|_R^{\frac{B+A}{B}}\leq e^{C_1\sqrt{\|W\|_{\Cc^1}}+C_2}\|u\|_{2R}^{\frac{A}{B}}\|u\|_{\frac{R}{2}},
\end{equation}
 
\noindent  Finally define  $\alpha=\frac{A}{A+B}$ and taking exponent $\frac{B}{A+B}$ of \eqref{last}: 
 \begin{equation*}
 \|u\|_R\leq e^{C_1\sqrt{\|W\|_{\Cc^1}}+C_2}\|u\|_{2R}^{\alpha}\|u\|^{1-\alpha}_{\frac{R}{2}}.
 \end{equation*} 

\end{proof}

From now on we assume that $M$ is compact. Thus we can derive from three balls theorem above uniform doubling estimate on solutions.

\begin{thm}[doubling estimate]\label{dor}

There exist two positive constants  $C_1$ and $C_2$,  depending only on  $M$ such that : if $u$ is a solution to   \eqref{sc} on $M$ with $W$ of class $\Cc^1$
then  for any $x_0$ in $M$ and any  $r>0$, one has
\begin{equation}\label{do}\|u\|_{B_{2r}(x_0)}\leq e^{C_1\sqrt{\|W\|_{\Cc^1}}+C_2}\|u\|_{B_r(x_0)}.
\end{equation}
\end{thm}
\vspace{0,5cm}
\begin{Rq}
Using standard elliptic theory to bound the $L^\infty$  norm of $|u|$ by a multiple of its $L^2$ norm, and rescaling in small ball gives for $\delta>0$ : 
$$\|u\|_{L^\infty(B_\delta(x_0))}\leq (C_1\|W\|_\infty+C_2)^{\frac{n}{2}}\delta^{-n/2}\|u\|_{L^2(B_{2\delta}(x_0))}.$$ 
Then  one can see that the doubling estimate is still true with the $L^\infty$ norm 
\begin{equation}
\|u\|_{L^\infty(B_{2r}(x_0))}\leq e^{C_1\sqrt{\|W\|_{\Cc^1}}+C_2}\|u\|_{L^\infty(B_r(x_0))}.
\end{equation}
\end{Rq}
\begin{Rq}
 We recall also that, it is necessary to assume that $M$ is compact to obtain an uniform upper bound on the vanishing order, and therefore a doubling estimate, on solutions. 
Indeed, consider the harmonic function, $f_k=\Re e(x_1+ix_2)^k$ defined in $\R^2$, so $f_k$ satisfies \eqref{sc} with $W=0$.
 The function $f_k$ can vanish at arbitrary high order at 0.
\end{Rq}

\noindent  To prove the theorem \ref{dor} we need to use the standard overlapping chains of balls argument (\cite{DF1,JL,K}) to show :  
 \begin{prop}\label{cor1}
For any $R>0$ their exists $C_R>0$ such that for any  $x_0\in M$, any $W\in \Cc^1(M)$ and any solutions $u$ to \eqref{sc} :
$$\| u\|_{B_R(x_0)}\geq e^{-C_R(1+\sqrt{\|W\|_{\Cc^1}})}\| u\|_{L^2(M)} .$$
\end{prop}

\begin{proof}
We may assume without loss of generality that $R<R_0$, with $R_0$ as in the three balls inequality (proposition \ref{ttc}).
 Up to multiplication by a constant, we can assume that $\| u\|_{L^2(M)}=1$.
 We denote by $\bar{x}$ a point in $M$ such that  $\| u\|_{B_R(\bar{x})}=\sup_{x\in M}  \|u\|_{B_R(x)}$.
 This implies that one has  $\| u\|_{B_{R(\bar{x})}}\geq D_R$, where $D_R$ depend only on $M$ and $R$. One has from proposition (\ref{ttc})  at an arbitrary point $x$ of $M$ : 

\begin{equation}\label{cop}\| u\|_{B_{R/2}(x)}\geq e^{-c(1+\sqrt{\|W\|_{\Cc^1}})}\| u\|^{\frac{1}{\alpha}}_{B_R(x)}.\end{equation} 

Let $\gamma$ be a geodesic curve  beetween $x_0$ and $\bar{x}$ and define  $x_1,\cdots,x_m=\bar{x}$ such that 
 $x_i\in\gamma$ and
 $B_{\frac{R}{2}}(x_{i+1})\subset B_R(x_i),$ for any  $i$ from $0$ to $m-1$. The number $m$  depends only on $\mathrm{diam}(M)$ and  $R$.
 Then the properties of $(x_i)_{1\leq i\leq m}$ and inequality \eqref{cop} give for all $i$, $1\leq i\leq m$ :
\begin{equation}
\|u\|_{B_{R/2}(x_i)}\geq e^{-c(1+\sqrt{\|W\|_{\Cc^1}})}\|u\|^{\frac{1}{\alpha}}_{B_{R/2}(x_{i+1})}.
\end{equation}

The result follows by iteration and the fact that $\| u\|_{B_R(\bar{x})}\geq D_R$.

\end{proof}
\begin{cor}\label{cor2}
For all $R>0$, there exists a positive constant $C_R$ depending only on $M$ and $R$ such that at any point $x_0$ in $M$ one has
\begin{equation*}
\| u\|_{R,2R}\geq e^{-C_R(1+\sqrt{\|W\|_{\Cc^1}})}\| u\|_{L^2(M)}.
\end{equation*}
\end{cor}
\begin{proof} 
 Recall that $\|u\|_{R,2R}=\|u\|_{L^2(A_{R,2R})}$ with $A_{R,2R}:=\{x\in M ; R\leq d(x,x_0)\leq 2R)\}$.
Let  $R<R_0$ where $R_0$ is from proposition \ref{3c1}, note that $R_0\leq \mathrm{diam}(M)$. 
Since $M$ is geodesically complete, there exists a point $x_1$ in  $A_{R,2R}$  
 such that $B_{x_1}(\frac{R}{4})\subset A_{R,2R}$. From proposition \ref{cor1} one has 
 $\|u\|_{B_{\frac{R}{4}}(x_1)}\geq e^{-C_R(1+\sqrt{\|W\|_{\Cc^1}})}\| u\|_{L^2(M)}$
 wich gives the result. 
\end{proof}
\begin{proof}[Proof of theorem \ref{dor}]

We proceed as in the proof of three balls inequality (proposition \ref{3c1}) except for the fact that now we want the first ball to become arbitrary small in front of the others.
 Let  $R=\frac{R_0}{4}$ with $R_0$ as in the three balls inequality,  let  $\delta$ such that  $0<3\delta<\frac{R}{8}$,
and  define  a smooth function $\psi$, with $0\leq\psi\leq1$  as follows: \vspace{0,3cm}
 
\begin{itemize}
\item[$\bullet$] $\psi(x)=0$ if $r(x)<\delta$ or if $r(x)>R$,
\item[$\bullet$] $\psi(x)=1$ if  $r(x)\in[\frac{5\delta}{4},\frac{R}{2}]$,
\item[$\bullet$] $|\nabla\psi(x)|\leq\frac{C}{\delta}$ and  $|\nabla^2\psi(x)|\leq\frac{C}{\delta^2}$ if $r(x)\in[\delta,\frac{5\delta}{4}]$ ,
 \item[$\bullet$] $|\nabla\psi(x)|\leq C$ and $|\nabla^2\psi(x)|\leq C$ if $r(x)\in[\frac{R}{2},R]$.\vspace{0,3cm}
\end{itemize}
Keeping appropriates terms in \eqref{Siv2} applied to $\psi u$ gives :  
\begin{equation*}
\begin{array}{rcl}
\|r^{\frac{\varepsilon}{2}}e^{\tau\phi}\psi u\|+ \tau\delta\|r^{-1}e^{\tau\phi}\psi u\|
&\leq &C\left(\|r^2e^{\tau\phi}\nabla u\cdot\nabla\psi\|+\|r^2e^{\tau\phi}\Delta\psi u\|\right). \vspace{0,15cm}\\
\end{array}
\end{equation*}
Using properties of $\psi$, one has

\begin{eqnarray*}
\|r^{\frac{\varepsilon}{2}}e^{\tau\phi} u\|_{\frac{R}{8},\frac{R}{4}}+\|e^{\tau\phi}u\|_{\frac{5\delta}{4},3\delta} \!\!&\leq  &\!\!C \left(\delta \|e^{\tau\phi}\nabla u\|_{\delta,\frac{5\delta}{4}}+\|e^{\tau\phi}\nabla u\|_{\frac{R}{2},R}\right)\\
&\!\!+\!\!&C\left(\|e^{\tau\phi} u\|_{\delta,\frac{5\delta}{4}}+\|e^{\tau\phi}u\|_{\frac{R}{2},R}\right).\nonumber
\end{eqnarray*}
 \noindent Using  (\ref{lm1}) and properties of   $\phi$, we get

\begin{equation*}\begin{array}{ccl}
e^{\tau\phi(\frac{R}{4})} \|u\|_{\frac{R}{8},\frac{R}{4}}&+&e^{\tau\phi(3\delta)}\|u\|_{\frac{5\delta}{4},3\delta} \\ 
&\leq& C(1+\|W\|_{\infty}^{1/2})\left(e^{\tau\phi(\delta)}\|u\|_{\frac{3\delta}{2}}+e^{\tau\phi(\frac{R}{5})}\|u\|_{\frac{5R}{3}}\right),\end{array}
\end{equation*}

\noindent and adding $e^{\tau\phi(3\delta)}\|u\|_{\frac{5\delta}{4}}$ to each side leads to 
\begin{equation*}\begin{array}{ccl}
e^{\tau\phi(\frac{R}{4})} \|u\|_{\frac{R}{8},\frac{R}{4}}&+&e^{\tau\phi(3\delta)}\|u\|_{3\delta}\\ &\leq &
C(1+\|W\|_{\infty}^{1/2})\left(e^{\tau\phi(\delta)}\|u\|_{\frac{3\delta}{2}}+e^{\tau\phi(\frac{R}{5})}\|u\|_{\frac{5R}{3}}\right).
\end{array}\end{equation*}
Now we want to choose $\tau$ such that  
$$C(1+\|W\|_{\infty}^{1/2}) e^{\tau\phi(\frac{R}{5})}\|u\|_{\frac{5R}{3}}\leq \frac{1}{2}e^{\tau\phi(\frac{R}{4})} \|u\|_{\frac{R}{8},\frac{R}{4}}.$$
For the same reasons than before we choose 
  $$\tau=\frac{1}{\phi(\frac{R}{5})-\phi(\frac{R}{4})}\mathrm{ln}\left(\frac{1}{2C(1+\|W\|_{\infty}^{1/2})}\frac{\|u\|_{\frac{R}{8},\frac{R}{4}}}{\|u\|_{\frac{5R}{3}}}\right)+C_1(1+\sqrt{\|W\|_{\Cc^1}}).$$
 Define $D_R=\left(\phi(\frac{R}{5})-\phi(\frac{R}{4})\right)^{-1}$; like before one has $0<A^{-1}\leq D_R\leq A$.
 Droping the first term in the left hand side, one has
  $$\|u\|_{3\delta}\leq e^{C(1+\|W\|_{\Cc^1})}\left(\frac{\|u\|_{\frac{R}{8},\frac{R}{4}}}{\|u\|_{\frac{5R}{3}}}\right)^{A}\|u\|_{\frac{3\delta}{2}} $$
  Finally from corollary \ref{cor2}, define  $r=\frac{3\delta}{2}$ to have : 
  $$\|u\|_{2r}\leq e^{C(1+\sqrt{\|W\|_{\Cc^1}})}\|u\|_{r}. $$
   Thus, the theorem is proved for all $r\leq\frac{R_0}{16}$.
 Using proposition \ref{cor1} we have for $r\geq \frac{R_0}{16}$ :
\begin{equation*}\begin{array}{rcl}
\|u\|_{B_{x_0}(r)}\geq\| u\|_{B_{x_0}(\frac{R_0}{16})}&\geq& e^{-C_0(1+\sqrt{\|W\|_{\Cc^1}})}\| u\|_{L^2(M)}\\
&\geq& e^{-C_1(1+\sqrt{\|W\|_{\Cc^1}})} \|u\|_{B_{x_0}(2r)}.
  \end{array}\end{equation*}
\end{proof}
As stated before, the upper bound on vanishing order of solutions (theorem \ref{van}) is a direct consequence of theorem \ref{dor} for non trivial solutions to \eqref{sc}.

\bibliographystyle{alpha}
\bibliography{4713.bib}

\end{document}